\documentclass[11pt]{amsart}
\usepackage{amsfonts, amsmath, amssymb, amscd, amsthm}
\usepackage{amsrefs}
\usepackage{mathtools}
\usepackage[utf8]{inputenc}
\usepackage{hyperref}
\usepackage{a4wide}
\theoremstyle{plain}
\newtheorem{thm}{Theorem}[section]
\newtheorem{cor}[thm]{Corollary}
\newtheorem{lem}[thm]{Lemma}
\newtheorem{defn}[thm]{Definition}

\newtheorem{rem}[thm]{Remark}

\newcommand{\ZZ}{\mathbb{Z}}

\newcommand{\CC}{\mathbb{C}}

\title{A class of semisimple Hopf algebras acting on quantum polynomial algebras}

\author{Deividi Pansera}

\address{Department of Mathematics, Faculty of Science, University of Porto, Rua Campo Alegre 687, 4169-007 Porto, Portugal}

\thanks{I would like to thank Christian Lomp very much for many clarifications and helpful insights. The author was partially supported by CMUP (UID/MAT/00144/2013), which is funded by FCT (Portugal) with national (MEC) and European structural funds through the programs FEDER, under the partnership agreement PT2020, and also supported by CAPES, Coordination of Superior Level Staff Improvement - Brazil.}
\begin{document}
	
	\begin{abstract}
		We construct a class of non-commutative, non-cocommutative, semisimple Hopf algebras of dimension $2n^2$  and present conditions to define an inner faithful action of these Hopf algebras on quantum polynomial algebras, providing, in this way, more examples of semisimple Hopf actions which do not factor through group actions. Also, under certain condition, we classify the inner faithful Hopf actions of the Kac-Paljutkin Hopf algebra of dimension $8$, $H_8$, on the quantum plane.
	\end{abstract}

	\maketitle
	
	\section{Introduction}
	
	Suppose that $H$ is a finite-dimensional Hopf algebra over a field $F$ acting on an algebra $A$. If $I$ is a Hopf ideal such that $I \cdot A = 0$, we say that the action \textit{factors through} a quotient Hopf algebra $H/I$. One says that the action \textit{factors through a group action} if there exists a Hopf ideal $I$ of $H$, with $I \cdot A =0$, such that $H/I \cong F[G]$ as Hopf algebra for some group $G$. In this last scenario, the Hopf action can be seen, in a certain sense, as a group action. Then, the following question arises: Are there conditions on either $H$ or $A$ ensuring that the action factors through a group action? 
	
	The first general result appeared in \cite{EtingofWalton}. Assuming that $H$ is semisimple and $F$ is algebraically closed, the authors showed that if $A$ is a commutative domain, then the action factors thorough a group action in this setting. Cuadra, Etingof and Walton, in \cite{CuadraEtingofWalton}, showed that that is also the case for the Weyl algebra $A = A_n(F)$, i.e., they showed that any semisimple Hopf action on the Weyl algebra $A_n(F)$ must factor through a group action. 
	
	In \cite{LompPansera}, by using and analyzing the results obtained by Cuadra, Etingof and Walton, it was showed that any semisimple Hopf action over an algebraically closed field of characteristic zero on an skew polynomial ring of \textit{derivation type} must factor through a group action. Although there are examples in literature of Hopf actions on quantum polynomial algebras that do not factor through group actions (see \cite[7.4-7.6]{kirkkuzhang}), in this paper we give some conditions to define an action of a Hopf algebra on an skew polynomial of \textit{automorphism type} which does not factor through a group action (Theorem \ref{actionquantu}). In order to do that, we will construct a class of semisimple Hopf algebras $H_{2n^2}$, which are not group algebras, and show that there exist inner faithful actions of those algebras on the quantum polynomial algebras, in particular on the quantum plane. In a recent paper, \cite{etingofwaltonquant}, P. Etingof and C. Walton say that there is \textit{no finite quantum symmetry} when the action of any finite-dimensional Hopf algebra factors through a group action. In this way, we give examples of algebras where there is quantum symmetry. Also, for the quantum plane case, under certain condition, we classify the inner faithful Hopf actions of the Kac-Paljutkin Hopf algebra of dimension $8$, $H_8$, on it (Theorem \ref{H8actionsQP}). 
	
	Let $F$ be a field. In this paper, all the Hopf algebras, tensor products and algebras are taken over $F$. In the sequel, we present the definition of \textit{inner faithful} actions. Let $H$ be a finite-dimensional Hopf algebra. A \textit{representation} of $H$ on an algebra $A$ is an algebra homomorphism $\pi: H \rightarrow A$. The following definition was given by Banica and Bichon.
	
	\begin{defn}[{\cite[Definition 2.7]{BanicaBichon}}]
		Let $\pi: H \rightarrow A$ be a representation of a Hopf algebra $H$ on an algebra $A$. We say that $\pi$ is inner faithful if $\operatorname{Ker}(\pi)$ does not contain any non-zero Hopf ideal.
	\end{defn}
	
	Let $M$ be a left $H$-module, then, if we consider the endomorphism algebra $\operatorname{End}_F(M)$, we have a representation $\pi: H \rightarrow \operatorname{End}_F(M)$ due to $M$ be an $H$-module. In this case, to say that $\pi$ is inner faithful is to say that $I \cdot M \neq 0$ for any non-zero Hopf ideal $I$ of $H$. This leads to the definition below.
	
	\begin{defn}[{\cite[Definition 1.2]{ChanWaltonZhang}}]
		Let $M$ be a left $H$-module. We say that $M$ is an inner faithful $H$-module (or $H$ acts inner faithfully on $M$) if $I\cdot M \neq 0$ for any non-zero Hopf ideal $I$ of $H$. Given a Hopf action of $H$ on an algebra $A$ (i.e., $A$ is a left $H$-module algebra), we say that this action is inner faithful if the left $H$-module $A$ is inner faithful.
	\end{defn}
	
	\begin{rem}
		Any Hopf action factors through an inner faithful Hopf action. For if $H$ is a Hopf algebra acting and a left $H$-module algebra $A$, then we can consider the Hopf ideal $I= \sum_{J \subseteq \operatorname{Ann}_H(A)} J$, which is the largest Hopf ideal of $H$ such that $I\cdot A=0$. Then $H/I$ acts inner faithfully on $A$.
	\end{rem}
	
	Also, we present here a result regarding group algebras which can be found in \cite[Lemma 4]{farnsteiner} and will be useful in the next sections. 
	
	\begin{lem}\label{normalideal}
		Let $H = F[G]$ be the group algebra of a finite group $G$ and $I \subsetneq H$ be a Hopf ideal. Then there exists a normal subgroup $N \lhd G$ such that $I = (F[G])(F[N])^+$.
	\end{lem}
	
	\section{A class of semisimple Hopf algebras}\label{skewpol}
	
	\subsection{Hopf algebra structure on quotients of skew-polynomial rings}
	Our first main results consist in giving necessary and sufficient condition to extend the structure of a biagebra $R$ to the skew polynomial ring of automorphism type $R[z;\sigma]$ (Theorem \ref{Jbialg}), and to define a Hopf algebra structure on the quotient  $R[z;\sigma]/I$, for $I$ a certain bi-ideal of $R[z;\sigma]$ (Theorem \ref{struchopf}). 
	
	We start with the definition of a twist for a bialgebra.

	\begin{defn}\label{righttwist}
		Let $R$ be a bialgebra and $J$ be an invertible element in $R \otimes R$. $J$ is called a \textit{right twist} (or a Drinfel'd twist) for $R$ if $J$ satisfies: 
		
		\begin{enumerate} 
			\item[(i)] $(id \otimes \Delta)(J)(1 \otimes J) = (\Delta \otimes id)(J)(J \otimes 1)$;
			\item[(ii)] $(id \otimes \epsilon)(J) = 1 = (\epsilon \otimes id)(J)$.
		\end{enumerate}
		
		$J$ is called a left twist for $R$ if it satisfies (ii) and $(1 \otimes J)(id \otimes \Delta)(J) = (J \otimes 1)(\Delta \otimes id)(J)$.
	\end{defn}
	If $J$ is a left twist for $R$, then $J^{-1}$ is a right twist for $R$.
	\begin{defn}[{\cite[2.1]{Davydov}}]\label{Jhom}
		Let $R$ be a bialgebra. Let $J$ be a right twist for $R$ and $\sigma \in \operatorname{End}(R)$. We say that the pair $(\sigma, J)$ is a twisted homomorphism for $R$ if $\sigma$ satisfies:
		
		\begin{enumerate}
			\item[(i)] $J(\sigma \otimes \sigma)\Delta(h) = \Delta(\sigma(h))J$ for all $h \in R$;
			\item[(ii)] $\epsilon \circ \sigma = \epsilon$.
		\end{enumerate}	
	\end{defn}
	
	Note that, for any homomorphism of coalgebras $\sigma \in \operatorname{End}(R)$, the pair $(\sigma, 1\otimes 1)$ is a twisted homomorphism for $R$.

	\begin{rem}
		Let $R$ be a bialgebra, $(\sigma, J)$ a twist homomorphism for $R$ and $a \in R$.
		\begin{enumerate}
			\item If $\sigma$ is an automorphism, then $(\sigma \otimes \sigma)\Delta(\sigma^{-1}(a)) = J^{-1}\Delta(a)J := \Delta^J(a)$;
			\item Since $J$ is invertible, $(\sigma \otimes \sigma) \Delta(a) = J^{-1}\Delta(\sigma(a))J = \Delta^J(\sigma(a))$;
			\item $\sigma$ is an homomorphism of coalgebras if and only if $J$ commutes with $\Delta(\sigma(a))$, for all $a \in R$.
		\end{enumerate}
	\end{rem}
	
	
	%
	
	Now, we shall extend the bialgebra structure of a bialgebra $R$ to its Ore extension of automorphism type $R[z; \sigma]$, for some specific $\sigma \in \operatorname{Aut}(R)$. 
	
	\begin{thm}\label{Jbialg}
		Let $R$ be a bialgebra and $(\sigma,J)$ be a twisted homomorphism for $R$. Let $H=R[z;\sigma]$ be the skew polynomial ring of endomorphism type. Then the bialgebra structure of $R$ can be extended to $H$ such that $\Delta(z) = J(z\otimes z)$ and $\epsilon(z)=1_F$. Conversely, if there exist an invertible element $J \in R \otimes R$ and $\sigma \in \operatorname{Aut}(R)$ such that  $R[z, \sigma]$ is a bialgebra with $\Delta(z) = J(z \otimes z)$ and $\epsilon(z) = 1$, then $(\sigma, J)$ is a twisted homomorphism for $R$. 
	\end{thm}
	
	\begin{proof}	
		Let $(\sigma, J)$ be a twisted homomorphism for the bialgebra $R$, and let $H = R[z;\sigma]$ be the Ore extension of automorphism type of $R$. Since $R$ is a bialgebra, we have a homomorphism of algebras $\Delta: R \rightarrow H \otimes H$. Consider the element $J(z \otimes z) \in H \otimes H$. Note that, for all $h \in R$, we have	
		\begin{equation*}
		J(z \otimes z)\Delta(h) = J(\sigma \otimes \sigma) \Delta(h)(z \otimes z) = \Delta(\sigma (h))J(z \otimes z).
		\end{equation*}	
		
		Thus $\Delta$ satisfies the Ore condition. Hence, there exists a unique algebra homomorphism $\overline{\Delta}: H \rightarrow H\otimes H$ such that $\overline{\Delta}_{|_R} = \Delta$ and $\overline{\Delta}(z) = J(z \otimes z)$. While it may be an abuse of notation, we just write $\overline{\Delta} = \Delta$.
		
		Furthermore, 
		
		\begin{align*}
		(id \otimes \Delta)\Delta(z) & = (id \otimes \Delta)(J)(z \otimes \Delta(z))\\
		&= (id\otimes \Delta)(J)(1\otimes J)(z \otimes z\otimes z)\\
		&=(\Delta \otimes id)(J)(J\otimes 1)(z \otimes z\otimes z) \\
		&= (\Delta \otimes id)(J)(\Delta(z) \otimes z) \\ 
		&= (\Delta \otimes id)\Delta(z);
		\end{align*} 
		This implies that $\Delta: H \rightarrow H \otimes H$ is a coassociative map.
		
		Now, since $\epsilon:R \rightarrow F$ is a homomorphism of algebras, then, for all $h \in R$, we have 
		$$1_F\epsilon(h) = 1_F\epsilon(\sigma(h)) = \epsilon(\sigma(h))1_F.$$
		
		Thus $\epsilon$ satisfies the Ore condition. So, there exists a unique algebra homomorphism $\overline{\epsilon}: H \rightarrow F$ such that $\overline{\epsilon}_{|_R} = \epsilon$ and $\overline{\epsilon}(z) = 1_F$. Again, while it may be an abuse of notation, we just write $\overline{\epsilon} = \epsilon$. Moreover, we have
		$$(id\otimes \epsilon)\Delta(z)= (id\otimes \epsilon)(J)z =  z= (\epsilon \otimes id)(J)z = (\epsilon \otimes id)\Delta(z).$$
		Thus $\epsilon$ satisfies the counity property in $H$. Therefore, the bialgebra structure of $R$ extends to $H$ as stated in the lemma.
		
		To prove the converse, suppose that there exist an invertible element $J \in R \otimes R$ and $\sigma \in \operatorname{Aut}(R)$ such that  $R[z, \sigma]$ is a bialgebra with $\Delta(z) = J(z \otimes z)$ and $\epsilon(z) = 1$. Since $(id \otimes \epsilon)\Delta(z) = (\epsilon \otimes id) \Delta(z)$, we must have that $(id \otimes \epsilon)(J) = 1 = (\epsilon \otimes id)(J)$. Also, $$(id \otimes \Delta)\Delta(z) = (id \otimes \Delta)(J)(z \otimes \Delta(z)) = (id\otimes \Delta)(J)(1\otimes J)(z \otimes z\otimes z)$$ and $$(\Delta \otimes id)\Delta(z) = (\Delta \otimes id)(J)(\Delta(z) \otimes z) = (\Delta \otimes id)(J)(J\otimes 1)(z \otimes z\otimes z).$$
		Since $R[z, \sigma]$ is a bialgebra, we have that $(id \otimes \Delta)\Delta(z) = (\Delta \otimes id)\Delta(z)$ and hence $(id\otimes \Delta)(J)(1\otimes J) = (\Delta \otimes id)(J)(J\otimes 1)$, that is, $J$ is a right twist for $R$. Moreover, for all $h \in R$, $zh = \sigma(h)z$. This implies that $$\epsilon(h) = \epsilon(zh) = \epsilon(\sigma(h)z) = \epsilon(\sigma(h)),$$i.e., $\epsilon \circ \sigma = \epsilon$. Note that $$J(\sigma \otimes \sigma)\Delta(h) (z \otimes z) = J(z \otimes z)\Delta(h) = \Delta(zh) = \Delta(\sigma(h)z) = \Delta(\sigma(h))J(z \otimes z), \quad \forall h \in R.$$ Thus, $J(\sigma \otimes \sigma)\Delta(h) = \Delta(\sigma(h))J$ and hence the pair $(\sigma, J)$ is a twisted homomorphism for $R$.
	\end{proof}
	
	As it was said at the beginning of this section, given a Hopf algebra $R$, we will find conditions to define a Hopf algebra structure on the quotient $R[z, \sigma]/I$, for some bi-ideal $I$ of $R[z, \sigma]$. The following lemma gives us certain conditions to find the bi-ideal on the bialgebra $R[z;\sigma]$ which will be used to define the Hopf algebra structure mentioned. 
	
	\begin{lem}\label{biideal}
		Let $R$ be a bialgebra and $(\sigma, J)$ be a twisted homomorphism for $R$. Suppose that there exists $0 \neq t \in R$ such that $\Delta(t) = J(\sigma \otimes \sigma)(J)(t \otimes t)$. Then, for $H=R[z;\sigma]$ with the bialgebra structure as in Theorem \ref{Jbialg}, $I = \langle z^2 - t \rangle$ is a bi-ideal of $H$. 
	\end{lem}
	\begin{proof}
		Let $R$ be a bialgebra and $(\sigma, J)$ be a twisted homomorphism for $R$. By Theorem \ref{Jbialg}, $H=R[z;\sigma]$ is also a bialgebra. Let $ 0\neq t \in R$ as in the hypothesis and let $I$ be the ideal in $H$ generated by $z^2 -t$. We have to prove that $I = \langle z^2 -t \rangle$ is coideal of $H$. We note that $t$ necessarily satisfies
		\begin{align*}
		t & = (\epsilon \otimes id)\Delta(t) \\
		& = (\epsilon \otimes id)(J)(\epsilon \otimes id)(\sigma \otimes \sigma)(J)\epsilon(t)t\\
		&=(\epsilon \otimes id)(\sigma \otimes \sigma)(J)\epsilon(t)t\\
		&= (\epsilon \otimes \sigma)(J)\epsilon(t)t \\
		&=  \sigma((\epsilon \otimes id)(J))\epsilon(t)t\\
		&= \sigma(1)\epsilon(t)t = \epsilon(t)t, 
		\end{align*} 
		which implies $\epsilon(t)=1$. So, $\epsilon(z^2 - t) = \epsilon(z)^2 - \epsilon(t) = 0$. That is, $I \subseteq Ker(\epsilon)$.
		
		Furthermore,
		\begin{align*}
		\Delta(z^2 - t) &= \Delta(z)^2 - \Delta(t) \\
		&= J(z \otimes z)J(z \otimes z) - \Delta(t) \\
		&= J(\sigma \otimes \sigma)(J)(z^2 \otimes z^2) - \Delta(t) \\
		&= J(\sigma \otimes \sigma)(J)(z^2 - t \otimes z^2) + J(\sigma \otimes \sigma)(J)(t \otimes z^2-t) + J(\sigma \otimes \sigma)(J)(t \otimes t) - \Delta(t) \\
		&= J(\sigma \otimes \sigma)(J)(z^2 - t \otimes z^2) + J(\sigma \otimes \sigma)(J)(t \otimes z^2-t), 
		\end{align*}
		which belongs to $I \otimes H + H \otimes I$. Therefore, $I$ is a bi-ideal of $H$. 
	\end{proof}
	
	\begin{rem}
		Note that, conversely, if $I = \langle z^2 - t \rangle$ is a bi-ideal of $H$, then $\Delta(t) - J(\sigma \otimes \sigma)(J)(t \otimes t) \in H \otimes I + I \otimes H$.
	\end{rem}
	
	Hence, given a bialgebra $R$ and $(\sigma, J)$ a twisted homomorphism for $R$ and an element $t$ that satisfies the hypothesis of Lemma \ref{biideal}, we have that $H/I$ is a bialgebra, for $H=R[z;\sigma]$ and  $I=\langle z^2 - t \rangle$. The next lemma presents conditions to extend a Hopf algebra structure from $R$ to the quotient bialgebra $H/I$. 
	
	\begin{thm}\label{struchopf}
		Let $R$ be a Hopf algebra with antipode $S$, $(\sigma,J)$ be a twisted homomorphism. Suppose also that $\sigma \circ S = S \circ \sigma$ and $\sigma^2 = id$. If there exists $ 0 \neq t \in R$, with $\Delta(t) = J(\sigma \otimes \sigma)(J)(t \otimes t)$, and such that
		\begin{enumerate}
			\item[(i)] $tJ^1S(J^2) = 1$;
			\item[(ii)] $t\sigma(S(J^1)J^2) = 1$,
		\end{enumerate}
		where $J = J^1 \otimes J^2$ with the summation omitted, then there exists a Hopf algebra structure on $H/I$ with $S(z)=z$. Conversely, if there exists a Hopf algebra structure on $H/I$ with $S(z)=z$, then $tJ^1S(J^2) = 1 = t\sigma(S(J^1)J^2)$. 
	\end{thm}
	\begin{proof}
		Let $R$ be a Hopf algebra with antipode $S$ and $(\sigma, J)$ a twisted homomorphism for $R$ such that $\sigma^2 = id$. By Theorem \ref{Jbialg}, $R[z,\sigma]$ is a bialgebra, and by Lemma \ref{biideal}, $I=\langle z^2 - t\rangle$ is a bi-ideal of $R[z,\sigma]$. 
		
		We just write $h$ for the element $h + I$ of $H/I$. And to define the antipode, we just extend the antipode $S$ of $R$ to $H/I$ defining $S(z)=z$. We note that $S: H/I \rightarrow H/I$ is well defined, since $S(z^2 - t) = S(z)^2 - S(t) = z^2 - t \in I$ and, using that $S \circ \sigma = \sigma \circ S$, $S(a)z = zS(\sigma(a))$ for all $a \in R$. Also, we have that
		\begin{equation*}
		\mu(id \otimes S)\Delta(z)= \mu(1 \otimes S)(J(z \otimes z)) =J^1zS(z)S(J^2)=z^2\sigma^2(J^1)S(J^2)=tJ^1S(J^2)= 1,
		\end{equation*}
		and
		\begin{equation*}
		\mu(S \otimes id)\Delta(z)= \mu(S \otimes 1)(J(z \otimes z))=S(z)S(J^1)J^2z =z^2\sigma(S(J^1))\sigma(J^2) =t\sigma(S(J^1)J^2) = 1.
		\end{equation*}
		
		Since $S$ is an antipode for $R$, the antipode property is verified for $R$ as well. Therefore, $S$ is an antipode of $H/I$ and so $H/I$ is a Hopf algebra. The converse follows from the two equations above. 
	\end{proof}
	
	In this setting, for $R$ a semisimple Hopf algebra, we have the following corollary. 
	
	\begin{cor}\label{semisimpleh2n}
		Under the conditions of Theorem \ref{struchopf}, if $R$ is a semisimple Hopf algebra, then $H/I$ is semisimple.
	\end{cor}

	\subsection{Construction of a class of semisimple Hopf algebras}\label{H2n2}
	
	In this subsection, using what we have done in the last subsection, we shall construct semisimple Hopf algebras of dimension $2n^2$, which, in the sequel, will be used to define actions on the quantum polynomial algebra which do not factor through group actions. From now on, we let the ground field $F$ be algebraically closed with characteristic zero.
	
	
	Let $\Gamma=\langle x \mid x^n=1 \rangle$ be the cyclic group of order $n>1$. Let $q\in F$ be a primitive $n$th root of unity. For every integer $j$, we set $$e_j = \frac{1}{n} \sum_{i=0}^{n-1} q^{-ij} x^i.$$
	
	Observe that if $j \equiv j' (\text{mod } n)$, then $q^j = q^{j'}$ and $x^j = x^{j'}$, and therefore $e_j = e_{j'}$. This means that $e_0, \hdots,e_{n-1}$ lists the distinct $e_i's$. Moreover, for  $0\leq j,k < n$, we have
	\begin{equation}\label{multx} e_jx^ k  = \frac{1}{n}\sum_{i=0}^{n-1} q^{-ij} x^{i+k} = 
	q^{jk} \left(\frac{1}{n}\sum_{i=0}^{n-1} q^{-(i+k)j} x^{i+k}\right) =q^{jk}e_j.\end{equation}
	
	\begin{lem}\label{completeset}
		$\{e_0, \cdots, e_{n-1}\}$ is a complete set of orthogonal idempotents of $F[\Gamma]$.
	\end{lem}
	\begin{proof}
		Since $q^{-j}$ is also an $n$th root of unity different from $1$ if $j\neq 0$, we get
		$$ \sum_{i=0}^{n-1} e_i = \frac{1}{n} \sum_{i=0}^{n-1}\sum_{j=0}^{n-1} q^{-ij}x^j
		= \frac{1}{n}\sum_{j=0}^{n-1}  \left( \sum_{i=0}^{n-1}\left(q^{-j}\right)^i \right) x^j = 1,$$
		
		Also, using (\ref{multx}), for $0\leq l, j < n$:
		$$ e_je_l = \frac{1}{n}\sum_{k=0}^{n-1} q^{-lk} e_jx^k = \frac{1}{n}\sum_{k=0}^{n-1} q^{-lk+jk}e_j = \frac{1}{n} \sum_{k=0}^{n-1} \left(q^{j-l}\right)^k e_j = \left\{\begin{array}{cc} e_j & \mbox{if } l=j\\ 0 & \mbox{if } l\neq j\end{array}\right.$$
		
		Hence, $\{e_0, \cdots, e_{n-1}\}$ is a complete set of orthogonal idempotents of $F[\Gamma]$.
	\end{proof}
	
	We denote the elements of $G = \Gamma \times \Gamma$ by $x^iy^s$ for $0\leq i,s < n$. Let $\{e_0, \cdots, e_{n-1}\}$ be the complete set of idempotents of $F[\Gamma]$ as in Lemma \ref{completeset}. Let $\sigma\in \operatorname{Aut}(F[G])$ be the automorphism of $F[G]$ induced by the group isomorphism $x^iy^s \mapsto x^sy^i$, for $1 \leq i,s \leq n$. 
	
	Set $\overline{e_i}:=\sigma(e_i)$, i.e., $\overline{e_i} = \frac{1}{n} \sum_{j=0}^{n-1}q^{-ij} y^j$. As in equation (\ref{multx}) one has  $\overline{e_i}y^ k  = q^{ik}\overline{e_i}$.
	
	Now, in $F[G] \otimes F[G]$, consider the element $J:= \sum_{i=0}^{n-1} e_i \otimes y^i$.
	Note that we can also write $J$ in terms of the elements $\overline{e_i}$'s as
	\begin{equation}\label{changingsides}J = \frac{1}{n}\sum_{i,j=0}^{n-1} q^{-ij} x^j\otimes y^i  = \sum_{i=0}^{n-1} x^i \otimes \overline{e_i} .\end{equation}
	
	With this setting, we have the following lemma. 
	
	\begin{lem}\label{twistedhomomorphism}
		The pair $(\sigma, J)$ is a twisted homomorphism for $F[G]$.
	\end{lem}
	\begin{proof}
		First we note that $J$ is invertible with inverse $J^{-1} =  \sum_{j=0}^{n-1} e_j \otimes y^{-j}$. Using (\ref{multx}) and (\ref{changingsides}), we get 
		\begin{eqnarray*}(\Delta\otimes 1)(J)(J\otimes 1) 
			&=&  \left(\sum_{i=0}^{n-1} \Delta(x^i)\otimes  \overline{e_i}\right)\left(
			\sum_{j=0}^{n-1} e_j \otimes  y^j \otimes 1\right)\\
			&=&  \sum_{i,j=0}^{n-1} x^ie_j \otimes x^iy^j \otimes  \overline{e_i}\\
			&=&  \sum_{i,j=0}^{n-1} q^{ij}e_j \otimes y^jx^i \otimes  \overline{e_i}\\
			&=&  \sum_{i,j=0}^{n-1} e_j \otimes y^jx^i \otimes  y^j\overline{e_i}\\
			&=&  \left( \sum_{j=0}^{n-1} e_j \otimes \Delta(y^j)\right)\left(\sum_{j=0}^{n-1} 1\otimes  x^i \otimes  \overline{e_i}\right)
			=  (1\otimes \Delta)(J)(1\otimes J).
		\end{eqnarray*}
		
		Moreover, we have $$(1\otimes \epsilon)(J)= \sum_{i=0}^{n-1} e_i = 1 = \sum_{i=0}^{n-1}\overline{e_i} = (\epsilon \otimes 1)(J).$$
		That is, $J$ is a right twist as in Definition \ref{righttwist}. For $x^ky^s \in G$, note that $$(\sigma \otimes \sigma)\Delta(x^ky^s) = (\sigma \otimes \sigma)(x^ky^s \otimes x^ky^s) = x^sy^k \otimes x^sy^k = \Delta(\sigma(x^ky^s)).$$ 
		Thus, since $F[G]$ is commutative, we have that $J(\sigma \otimes \sigma)\Delta(x^ky^s) = \Delta(\sigma(x^ky^s))J$. Moreover, clearly $\epsilon \circ \sigma = \epsilon$. Therefore, the pair $(\sigma, J)$ is a twisted homomorphism for $F[G]$.
	\end{proof}
	
	Hence, by Theorem \ref{Jbialg}, $H = F[G][z;\sigma]$ is bialgebra with $\Delta(z)=J(z\otimes z)$ and $\epsilon(z)=1$. Now, consider the element $t=\sum_{i=0}^{n-1} e_i y^i$ which satisfies
	$$\sigma(t)=\sum_{i=0}^{n-1} \overline{e_i}x^i  = \frac{1}{n} \sum_{i=0}^{n-1}\sum_{j=0}^{n-1} q^{-ij}y^jx^i = \sum_{j=0}^{n-1} \left(\frac{1}{n} \sum_{i=0}^{n-1}q^{-ij}x^i\right)y^j = \sum_{j=0}^{n-1} e_j y^j = t.$$
	
	Moreover, $t$ has an inverse  in $F[G]$, $t^{-1}= \sum_{i=0}^{n-1} e_iy^{-i}$.
	
	\begin{lem}\label{tsatisfies}
		The element $t$ satisfies $\Delta(t) = J(\sigma \otimes \sigma)(J)(t \otimes t)$.
	\end{lem}
	\begin{proof}
		Since $\sum_{k=0}^{n-1} (q^{(l-j)})^k = 0$ for $l\neq j$, note that for any $i$
		\begin{align*}
		\sum_{k=0}^{n-1} e_k \otimes e_{(i-k)} & =  \dfrac{1}{n^2} \sum_{k,j,l=0}^{n-1} q^{-jk}q^{-l(i-k)}x^j \otimes x^l \\
		& =  \dfrac{1}{n^2} \sum_{k,j,l=0}^{n-1} q^{(l-j)k}q^{-li}x^j \otimes x^l \\
		& =  \dfrac{1}{n^2} \sum_{j,l=0}^{n-1} \left(\sum_{k=0}^{n-1}(q^{(l-j)})^k\right)q^{-li}x^j \otimes x^l \\
		& =  \dfrac{1}{n} \sum_{j=0}^{n-1} q^{-ij}x^j \otimes x^j  =  \Delta(e_i).
		\end{align*}	
		So, it follows that $\Delta(t) = \sum_{i,m=0}^{n-1}e_iy^{m} \otimes e_{(m-i)}y^{m}$. Then, by (\ref{changingsides}), we get
		\begin{align*}
		J(\sigma \otimes \sigma)(J)(t \otimes t) & = \sum_{i,j=0}^{n-1}e_iy^jt \otimes y^ie_jt \\
		& = \sum_{i,j,k,l=0}^{n-1}e_ie_ky^{j+k} \otimes e_je_ly^{i+l} \\
		& = \sum_{i,j=0}^{n-1}e_iy^{i+j} \otimes e_jy^{i+j}  \\
		& = \sum_{i,m=0}^{n-1}e_iy^{m} \otimes e_{(m-i)}y^{m} = \Delta(t).
		\end{align*}
	\end{proof}
	So, $t$ satisfies the hypothesis of Lemma \ref{biideal} and hence $I=\langle z^2-t\rangle$ is a bi-ideal of $F[G][z;\sigma]$. Thus $H=F[G][z;\sigma]/\langle z^2-t\rangle$ is also a bialgebra. 
	
	Note that $\sigma^2 = id$ and $$S(t) = \frac{1}{n}\sum_{i,j=0}^{n-1}q^{-ij}x^{-i}y^{-j} = \frac{1}{n}\sum_{k,s=0}^{n-1}q^{-ks}x^{k}y^{s} = \sum_{s=0}^{n-1}e_sy^{s}= t.$$
	
	Since $\sigma(S(x^ky^s)) = x^{-s}y^{-k} = S(\sigma(x^ky^s))$ and $$tJ^{(1)}S(J^{(2)}) = \sum_{i,j=0}^{n-1}e_ie_jy^{i-j} = \sum_{i=0}^{n-1}e_i = 1,$$
	and
	$$t\sigma(S(J^{(1)})J^{(2)}) = t\sigma\left(\frac{1}{n}\sum_{i,j=0}^{n-1}q^{-ij}x^{-j}y^{i}\right) = t\sum_{j=0}^{n-1}y^{-j}\frac{1}{n}\sum_{i=0}^{n-1}q^{-ij}x^{i} = t\left(\sum_{j=0}^{n-1}e_jy^{-j}\right)=tt^{-1} = 1,$$
	by Theorem \ref{struchopf}, $H = F[G][z;\sigma]/\langle z^2 - t\rangle$ is a Hopf algebra. 
	
	Note that these Hopf algebras have dimension $2n^2$ and we shall denote them by $H_{2n^2}$. 
	
	Also, we note that since $F$ has characteristic zero, $F[G]$ is semisimple, and then $H_{2n^2}$ is also semisimple. Moreover, $H_{2n^2}$ is non-commutative and non-cocommutative. 
	
	Before we continue, we establish the following lemma about the Hopf ideals of $H_{2n^2}$, which will be useful in the sequel. 
	
	\begin{lem}\label{hopfidealsh2n2}
		Let $I$ be a Hopf ideal of $H_{2n^2}$. If $I \cap R = 0$, then $I = 0$. 
	\end{lem}
	\begin{proof}
		Let $I$ be a Hopf ideal of $H_{2n^2}$ such that $I \cap R = 0$. Consider the restriction of the projection map $\pi_{|_R}: R \rightarrow H_{2n^2}/I$. Clearly, $\operatorname{Ker}\pi_{|_R} = I \cap R$. Hence, we can look at $R$ as a Hopf subalgebra of $H_{2n^2}/I$. Then $\dim(R)$ divides $\dim(H_{2n^2}/I)$. 
		
		Since $\dim(H_{2n^2}) = 2 \dim(R)$, we must have that $\dim(H_{2n^2}/I) = 2\dim(R)$ or $\dim(H_{2n^2}/I) = \dim(R)$. If $\dim(H_{2n^2}/I) = 2\dim(R)$, then $I = 0$. 
		
		If $\dim(H_{2n^2}/I) = \dim(R)$, then $H_{2n^2}/I \cong R$ and thus $H_{2n^2}/I$ is commutative. Hence, we must have that $\bar{x}\bar{z}=\bar{z}\bar{x}=\bar{y}\bar{z}$, which implies that $(x-y)z \in I$. So, $(x-y)z(zt^{-1}) = (x-y) \in I \cap R = 0$, which is absurd.
	\end{proof}
	

	
	\section{Inner faithful actions on quantum polynomial algebras}\label{qpaaction}
	
	Let $M=(m_{ij}) \in M_{r\times r}(F^{\times})$ be a square matrix of size $r$ such that $m_{ii} = m_{ij}m_{ji}=1$. Let $A_M = F_M[u_1, \hdots, u_r]$ be the \textit{quantum polynomial algebra} (see  \cite{artamanov} and \cite[Appendix I.14 and Chapter I.2]{brown}), i.e., the associative $F$-algebra generated by $u_1, \hdots, u_r$ subject to the relations $$u_iu_j = m_{ij}u_ju_i, \quad 1 \leq i,j \leq r.$$
	Alternatively, quantum polynomial algebras can be constructed as iterated Ore extension of automorphism type $$A_M = F[u_1][u_2, \tau_2]\cdots[u_r, \tau_r],$$
	where $\tau_i(u_j)=m_{ij}u_j$ for all $i,j$ with $1 \leq j < i \leq r$. 
	
	In this section, we will present inner faithful actions of $H_{2n^2}$ on $A_M$. Our main result is Theorem \ref{actionquantu}, where we provide conditions to define  inner faithful actions of $H_{2n^2}$ on quantum polynomial algebras which do not factor through group actions.  
	
	Recall that for a $n$th primitive root of unity $q$, 
	\begin{equation*}
	\Delta(z) = J(z \otimes z) = \frac{1}{n}\sum_{i,j=0}^{n-1}q^{-ij}x^iz \otimes y^jz.
	\end{equation*}

	\begin{thm}\label{actionquantu}
		Let $n>1$ be a natural number and $H_{2n^2}$ the Hopf algebra constructed above with respect to a primitive $n$th root of unity  $q$.
		Let $A_M$ be a quantum polynomial algebra in generators $u_1, \ldots, u_r$  with respect to a matrix $M=(m_{ij}) \in M_{r\times r}(F)$.
		For any  permutation $\tau\in S_r$ and elements $\lambda, \mu \in F^r$   the following holds:
		
		\begin{enumerate}
			\item $H_{2n^2}$ acts on $A_M$ via $x \cdot u_i = \lambda_{i} u_i$, $y\cdot u_i =\mu_i u_i$ and $z\cdot u_i = u_{\tau(i)}$ if and only if there exists $b\in \ZZ_n^r$ such that
			\begin{equation*}\label{condition}\lambda_{i}=q^{b_{i}}, \qquad \mu_i = q^{b_{\tau(i)}}, \qquad  m_{st}=q^{B_{\tau(s)\tau(t)}}m_{\tau(s)\tau(t)},\end{equation*}
			for all $j,i,s$, where $B_{st}=b_{s}b_{\tau(t)}-b_{t}b_{\tau(s)}$.
			\item The action in (1) is inner faithful if and only if the map $f:\ZZ_n^2 \rightarrow \ZZ_n^r$ with $f(i,j)=(ib_1+jb_{\tau(1)}, \cdots, ib_r+jb_{\tau(r)})$ is injective.
			\item In particular if there exist $s,t$ such that  $B_{st}$ is invertible in $\ZZ_n$, then the action in (1) is inner faithful.
		\end{enumerate}
	\end{thm}
	\begin{proof}
		$(1)$ Suppose $H_{2n^2}$ acts on $A_M$ via
		$x \cdot u_i = \lambda_{i} u_i; \quad y \cdot u_i = \mu_{i} u_i; \quad z \cdot u_i = u_{\tau(i)}$.
		Since $zy = xz$ in $H_{2n^2}$, we should have that $(zy)\cdot u_i = (xz)\cdot u_i$, for all $i \in \{1, \hdots, r\}$. This happens if and only if $\mu_{i} = \lambda_{\tau(i)}$ for all $i \in \{1, \hdots, r\}$.
		Since $x^n = 1$, we must have that $\lambda_{i}$ is an $n$th roots of unity and hence there exist $b\in \ZZ_n^r$ such that $\lambda_i=q^{b_i}$ and  $\mu_i=\lambda_{\tau(i)}=q^{b_{\tau(i)}}$ for all $i$.
		
		Let us calculate $z\cdot (u_ku_l)$ for some indices $k,l$:
		\begin{eqnarray*}
			z\cdot(u_ku_l) &=& \sum_{i,j=0}^{n-1} q^{-ij}(\lambda^i_{\tau(k)}u_{\tau(k)})(\mu^j_{\tau(l)}u_{\tau(l)})\\
			&=& \left(\sum_{i,j=0}^{n-1} q^{-ij+b_{\tau(k)}i+b_{\tau(\tau(l))}j}\right)u_{\tau(k)}u_{\tau(l)}\\
			&=& \left(\sum_{i=0}^{n-1} q^{b_{\tau(k)}i} \sum_{j=0}^{n-1} (q^{b_{\tau(\tau(l))}-i})^j\right)u_{\tau(k)}u_{\tau(l)}  = n q^{b_{\tau(k)\tau(\tau(l))}} u_{\tau(k)}u_{\tau(l)}
		\end{eqnarray*}
		
		Hence the relation $z \cdot (u_ku_l) = m_{kl} z \cdot (u_lu_k)$ holds if and only if:
		$$m_{\tau(k)\tau(l)} q^{b_{\tau(k)}b_{\tau(\tau(l))}} = m_{kl} q^{b_{\tau(l)}b_{\tau(\tau(k))}} $$
		or equivalently $m_{kl} = m_{\tau(k)\tau(l)} q^{B_\tau(k)\tau(l)}  m_{\tau(k)\tau(l)}$, with $B_{\tau(k)\tau(l)} = b_{\tau(k)}b_{\tau(\tau(l))}-b_{\tau(l)}b_{\tau(\tau(k))}$.
		
		Conversely if $\lambda, \mu$ and $(m_{ij})$ satisfy the conditions indicated above, then the action given as indicated is well-defined as $1-x^n$, $1-y^n$, $zx-yz$ act as zero on the $u_i$'s and as $z \cdot (u_ku_l) = m_{kl} z \cdot (u_lu_k)$ holds.
		
		$(2)$ Let $I$ be a Hopf ideal of $H_{2n^2}$ such that $I \cdot A_M = 0$. Suppose that $I \cap R \neq 0$. Since $I \cap R$ is a Hopf ideal of the group algebra $R$, we can apply Lemma \ref{normalideal} and conclude that there exists a normal subgroup $N$ of $\ZZ_n \times \ZZ_n$ such that $I \cap R = RF[N]^+$. Since $I \cap R \neq 0$ we have $N \neq \langle (0,0)\rangle$, i.e. there exist $(i,j) \neq (0,0)$, such that $x^iy^j \in N$ and hence $1-x^iy^j \in I$. Thus, for all $s \in \{1, \hdots, r\}$, we have 
		$(1-q^{b_si + b_{\tau(s)}j})u_s = (1-x^iy^j)\cdot u_s =0,$ which implies $ib_s + jb_{\tau(s)} = 0$ in $\ZZ_n$  for all $s \in \{1, \hdots, r\}$. Hence if $f$ is injective, then  $I \cap R = 0$, and thus, by Lemma \ref{hopfidealsh2n2}, $I=0$. So, the action must be inner faithful. Conversely if $f$ is not injective and $f(i,j)=(0,\ldots, 0)$ for some $(i,j)\in\ZZ_n^2$, then the Hopf ideal generated by  $1-x^iy^j$ annihilates $A_M$ and the action is not inner faithful.
		
		$(3)$ Suppose that $B_{st}$ is  invertible  for some $s,t$. If 
		$(i,j)$ is a pair such that $ib_k + jb_{\tau(k)}=0$ for all $k$, then in particular 
		$b_{\tau(s)}j=-b_s i$ and hence $j B_{st} = jb_sb_{\tau(t)} - jb_{\tau(s)}b_t = -b_sb_t i  + b_sb_ti = 0$. Thus $j=0$ and  therefore $ib_k=0$ for all $k$. In particular $iB_{st}=0$ and hence $i=0$.

	\end{proof}

	Suppose $F$ contains a primitive $6$th root of unity $q$. Let $H_{72}$ be the semisimple Hopf algebras with respect to $q$ and $n=6$.
	Let $A_M=F_{q^2}[u_1,u_2][u_3]$ be the quantum polynomial algebra in the generators $u_1, u_2, u_3$ subject to $u_2u_1=q^2u_2u_1$ and $u_3$ commuting with $u_1$ and $u_2$. Then there exists an inner faithful action of $H_{72}$ on $F_{q^2}[u_1,u_2][u_3]$ given by $b=(2,2,1)\in \ZZ_6^3$ and $\tau=(13)\in S_3$, i.e. 
	$$x\cdot u_1 = q^2 u_1,\qquad x\cdot u_2 = q^2u_2, \qquad x\cdot u_3=qu_3$$ 
	$$y\cdot u_1 = q u_1,\qquad y\cdot u_2 = q^2u_2, \qquad y\cdot u_3=q^2 u_3$$
	$$z\cdot u_1 = u_3, \qquad z\cdot u_2 = u_2, \qquad z\cdot u_3 = u_1.$$
	Note that the function $f:\ZZ_6^2 \rightarrow \ZZ_6^3$ from Theorem \ref{actionquantu}(2) is given by 
	$f(i,j)=(2i+j, 2(i+j), i+2j)$ and injective. However the elements $B_{st}$ from above are all zero divisors in $\ZZ_6$.

	\subsection{Inner faithful actions  on the quantum plane}
	
	For a moment, let $F$ be an arbitrary field. Let $ 0 \neq p \in F$ and consider the matrix $$M=\left(\begin{array}{cc} 1&p^{-1} \\ 
	p& 1\end{array}\right) \in M_{2\times 2}(F)$$
	
	The \textit{quantum plane} is the quantum polynomial algebra $A_M = F_M[u,v]$ with two generators. We denote the quantum plane $F_M[u,v]$ by $F_p[u,v]$ and we call $p$ a \textit{parameter}. When $p \neq 1$ the algebra $F_p[u,v]$ is non-commutative. 
	


	Consider the Hopf algebras $H_{2n^2}$ as constructed in Section \ref{H2n2}. For each $n$, these Hopf algebras act on $A=\CC_p[u,v]$ with $p^2 = q$, where $q$ is the primitive $n$th root of unity used to construct $H_{2n^2}$. The action is given by: 
	\begin{align*}
	x\cdot u = qu, &  & y \cdot u = u,  & & z \cdot u = v, \\  
	x\cdot v = v, &  & y \cdot v = qv,  & & z \cdot v = u,
	\end{align*}
	which corresponds to $\tau = (12) \in S_2$ and the vector $ b = (1,0) \in \ZZ_2^2$ in Theorem \ref{actionquantu}. $H_{2n^2}$ acts on $A$ and, since the element $B_{12}$ is invertible in $\ZZ_2$, by Theorem \ref{actionquantu}, this action is inner faithful. 

	\section{Inner faithful actions of Kac-Paljutkin Hopf algebra on the quantum plane}
	G.I. Kac and Paljutkin, in the $1960$'s, discovered a non-commutative, non-cocommutative semisimple Hopf algebra of dimension $8$ \cite{KacPaljutkin}.  $H_8$ is the algebra over $F$ generated by $x,y$, and $z$ subject to the following relations 
	\begin{eqnarray*}
		x^2 = 1, &  y^2 = 1, & xy=yx \\
		z^2 = \frac{1}{2}\left(1+x+y-xy\right), & zx = yz, & zy=xz. 
	\end{eqnarray*}
	
	$H_8$ has a coalgebra structure with
	\begin{eqnarray*}
		& \Delta(x)=x \otimes x, & \epsilon(x) = 1 \\
		& \Delta(y)=y \otimes y, & \epsilon(y) = 1 \\
		& \Delta(z)=\frac{1}{2}\left(1 \otimes 1 + x \otimes 1 + 1 \otimes y - x \otimes y \right)(z \otimes z), & \epsilon(z) = 1. 
	\end{eqnarray*}
	$H_8$ becomes a Hopf algebra by setting $S(x)= x$, $S(y)=y$, and $S(z) = z$.
	
	Later, in $1995$, A. Masouka showed that there is only one (up to isomorphisms) semisimple Hopf algebra of dimension $8$ ($H_8$) that is neither commutative nor cocommutative \cite{masuokah8}.  Although Kac and Paljutkin present $H_8$ as above and Masouka presents it under the perspective of biproducts and bicrossed products, in this section we present $H_8$ as the Hopf algebra $H_{2n^2}$, for $n =2$. Also, under a certain condition, we classify the inner faithful actions of $H_8$ on the quantum plane. 
	\subsection{$H_8$ as a quotient of an Ore extension}
	
	For the Hopf algebras constructed in that Section \ref{H2n2}, we take $n=2$ and $q=-1$. The group $\Gamma$, then, is the cyclic group of order $2$, $\mathbb{Z}_2 = \{x \mid x^2 = 1\}$. In this setting, we have that the orthogonal idempotents of $F[\mathbb{Z}_2]$ as in Lemma \ref{completeset} are given by: $$e_0 = \frac{1}{2}(1+x) \, \text{ and } \, e_1 = \frac{1}{2}(1-x).$$
	
	Then, for $G = \langle x,y \mid x^2 = 1 = y^2, xy=yx \rangle = \Gamma \times \Gamma$, the automorphism $\sigma$ swaps $x$ and $y$, i.e., $\sigma(x) = y$ and $\sigma(y) = x$.  And the element $J$ is given by $$J = \frac{1}{2}((1+x)\otimes 1 + (1-x) \otimes y)= \frac{1}{2}(1 \otimes 1 + x \otimes 1 + 1 \otimes y - x \otimes y).$$
	
	Let $R = F[G]$. So, $R[z;\sigma]$ becomes a bialgebra with $$\Delta(z) = J(z \otimes z) = \frac{1}{2}\left(1 \otimes 1 + x \otimes 1 + 1 \otimes y - x \otimes y\right)(z \otimes z),$$
	and $\epsilon(z)=1$. Also, note that $zx=\sigma(x)z = yz$. Since $t$ is given by $$t = e_0 + e_1y = \frac{1}{2}(1 + x + y - xy),$$ 
	we get that $z^2 = \frac{1}{2}(1 + x + y - xy)$ in the Hopf algebra $R[z;\sigma]/ \langle z^2 - t \rangle$, where $S(z)=z$. 
	
	So, $H_{2n^2}$, for $n=2$, is precisely the Hopf algebra $H_8$. Then, from now on, every time we refer to $H_8$, we keep in mind its presentation as the one presented in this subsection, i.e., as a quotient of an Ore extension: $H_8 = R[z;\sigma]/\langle z^2 - t \rangle$.  
	
	\subsection{Classification of the Inner faithful actions of $H_8$ on the quantum plane}
	
	Let $F = \CC$ and $A=\CC_p[u,v]$ be the quantum plane with parameter $p\in \CC^\times$, i.e., $vu=puv$. 
	
	In the following theorem we classify the possibles inner faithful actions of $H_8$ on $A$ under a certain assumption.
	
	\begin{thm}\label{H8actionsQP}
		Let $p \in \CC^{\times}$. If there is a Hopf action of $H_8$ on the quantum plane $A=\CC_p[u,v]$ such that $z \cdot u = v$ and $z \cdot v = u$, then this action is inner faithful and $p^2= -1$. 
	\end{thm}
	\begin{proof}
		We can assume $p\neq 1$, because if  $p=1$, then $A = \CC[u,v]$ is the commutative polynomial ring, which is a commutative domain. Therefore, Etingof and Walton's result \cite[Theorem 1.3]{EtingofWalton} guarantee that there cannot be any inner faithful action of $H_{8}$ on $A$, since $H_{8}$ is not a group algebra. 
		
		If there is an action of $H_8$ on $A$, since $x$ and $y$ are group-like elements, they act as automorphisms of $A$. Hence there exist $\alpha,\beta\in \operatorname{Aut}(A)$ such that $x\cdot a = \alpha(a)$ and $y\cdot a = \beta(a)$ for all $a\in A$. Also, since $x^2 = 1 = y^2$, $\alpha^2 = id = \beta^2$. 
		
		Under the assumption that $z$ acts by interchanging $u$ and $v$, we must have $z\cdot(vu) = pz\cdot(uv)$ or equivalently
		\begin{equation}\label{eq1a}uv+\alpha(u)v +u\beta(v)-\alpha(u)\beta(v) = p ( vu + \alpha(v)u + v\beta(u) - \alpha(v)\beta(u)).\end{equation}	
		
		Moreover, since $xz=zy$, it follows that $(xz)\cdot u = (zy)\cdot u$ and $(xz)\cdot v = (zy)\cdot v$, which implies that 	\begin{equation}\label{important}
		\alpha(u) = z \cdot \beta(v) \quad \textrm{and} \quad \alpha(v) = z \cdot \beta(u).
		\end{equation}
		
		Now, we separate the proof in cases. 	
		
		{\bf CASE I: $p \neq -1$:} In this case Alev and Chamarie showed in \cite[1.4.4]{AlevChamarie} that the automorphisms of $A$ are given by a torus action, i.e., $x$ and $y$ acts as scalars on $u$ and $v$. Hence we are in the situation of Theorem \ref{actionquantu} $(1)$ with $\tau = (12) \in S_2$. Suppose $$x\cdot u = (-1)^{b_1}u, \qquad  x\cdot v = (-1)^{b_2}v, \qquad y\cdot u = (-1)^{b_{\tau(1)}}u, \qquad y\cdot v = (-1)^{b_{\tau(2)}}v,$$and let $b = (b_1,b_2) \in \ZZ_2^2$. By Theorem \ref{actionquantu}, $H$ acts on $A$ if and only if $$p^2 = (-1)^{b_1b_{\tau(2)} - b_2b_{\tau(1)}} = (-1)^{B_{12}}.$$Since $p \neq -1$, we must have $B_{12} = 1$ and so $p^2 = -1$. Hence, by Theorem \ref{actionquantu} $(3)$, the action is inner faithful.
		%
		%
		
		\medskip
		
		{\bf CASE II: $p=-1$:} In this case Alev and Chamarie showed in \cite[1.4.4]{AlevChamarie} that $\operatorname{Aut}(A)$ is a semidirect product of $(\CC^\times)^2$ with the cyclic group of order $2$ given by the automorphism $\tau$ that flips $u$ and $v$. Hence any automorphism of $A$ is either an element of $(\CC^\times)^2$ or a product of an element $(\alpha_0,\alpha_1)\in (\CC^\times)^2$ and $\tau$. By \ref{important}, if $\beta$ is just given by a torus action, then $\alpha$ has to be also given by a torus action, and if $\beta$ is given by a torus action and $\tau$, then $\alpha$ has to be also given by a torus action and $\tau$. Therefore, these are the only two possible cases for $\alpha$ and $\beta$. 
		
		{\bf CASE II.a}: If $\beta$ is given only by a torus action, then then we are in the same situation as CASE I and hence, by what we have done for CASE I, necessarily $p^2 = -1$, which contradicts $p=-1$. So, $\beta$ cannot be given by a torus action.
		
		%
		
		
		{\bf CASE II.b}: Suppose that both $\alpha$ and $\beta$ are compositions of a torus action and $\tau$, then there are $(\alpha_0,\alpha_1),(\beta_0,\beta_1)\in (\CC^\times)^2$ such that
		$$x\cdot u = \alpha(u)=\alpha_0v, \qquad x\cdot v = \alpha(v)=\alpha_1u, \qquad y\cdot u = \beta(u)=\beta_0v, \qquad y\cdot v = \beta(v)=\beta_1u.$$
		Then equation (\ref{eq1a}) yields
		$$uv+\alpha_0v^2+\beta_1u^2 + \beta_1\alpha_0uv = uv - \alpha_1u^2 -  \beta_0v^2 + \beta_0\alpha_1uv,$$
		which is equivalent to
		$$  (\alpha_1+\beta_1)u^2+(\alpha_0+\beta_0)v^2 + (\beta_1\alpha_0-\beta_0\alpha_1)uv = 0$$
		and implies $\beta_i= - \alpha_i$, for $i=0,1$, and $\beta_1\alpha_0 = \beta_0\alpha_1$. Now, since $\alpha^2 = id = \beta^2$, we must have that $\alpha_0\alpha_1=1$ and $\beta_0\beta_1 = 1$. Also, by \ref{important}, we must have that $\alpha_1 = \beta_0$ and $\alpha_0 = \beta_1$. Hence, $\alpha_0^2 = -\alpha_0\alpha_1 = -1$. And so $\alpha_0 = \pm i$. Therefore, for $\alpha_0 \in \{-i,i\}$, the options are $\beta_1 = \alpha_0$, $\beta_0 = - \alpha_0$, $\alpha_1 = -\alpha_0$. But note that $$xz \cdot (vu) = (vu - \alpha_0u^2 - \alpha_0v^2 - vu) \quad \textrm{ and } \quad zy \cdot (vu) = (-uv - \alpha_0 u^2 - \alpha_0 v^2 - uv).$$
		This leads to a contradiction, since $xz=yz$. Therefore, there can be no action at all for the case where $p=-1$.
	\end{proof}

\end{document}